\documentclass[a4paper,12pt]{amsart}
\usepackage[letterpaper, margin=1.2in]{geometry}
\usepackage{latexsym}
\usepackage{amssymb}
\usepackage{amsmath}
\usepackage{amsfonts}
\usepackage{color}
\usepackage{tikz}
\usepackage[normalem]{ulem }
\usepackage{soul}
\newtheorem{thm}{Theorem}[section]

\newtheorem{lem}[thm]{Lemma}

\newtheorem{cor}[thm]{Corollary}

\theoremstyle{definition}

\newtheorem*{rem}{Remark}

\newcommand{\R}{\mathbb R}

\newcommand{\Z}{\mathbb Z}
\newcommand{\q}{\mathbb Q}

\newcommand{\F}{\mathbb F}

\newcommand{\p}{\mathfrak{p}}
\newcommand{\Q}{\mathfrak{Q}}
\def\ph{\phi}
\def\al{\alpha}

\def\s{\sigma}
\def\d{\delta}
\def\t{\theta}
\def\om{\omega}

\def\Ga{\Gamma}
\def\ol{\overline}

\def\ol{\overline}
\newcounter{cs}
\stepcounter{cs}
\newcommand{\casos}{\begin{itemize}}
\newcommand{\fcasos}{\end{itemize}\setcounter{cs}{1}}

\newfont{\tit}{cmr12 scaled \magstep3}
\begin{document}
\title[]{On a Theorem of Dedekind}
\author{Lhoussain El Fadil}
\address[L. El Fadil]{Department of Mathematics, Faculty of Sciences Dhar-El Mahraz, University of Sidi Mohamed Ben Abdellah, P.O.B. 1796, Fez, Morocco} \email{lhouelfadil2@gmail.com}
\author{Abdulaziz Deajim}
\address[A. Deajim]{Department of Mathematics
King Khalid University
P.O.Box 9004, Abha
Saudi Arabia} \email{deajim@kku.edu.sa, deajim@gmail.com}
\author{Ahmed Najim}
\address[A. Najim]{Department of Mathematics, Faculty of Sciences Dhar-El Mahraz, University of Sidi Mohamed Ben Abdellah, P.O.B. 1796, Fez, Morocco}{ \email{najimed9@yahoo.fr }}
\subjclass[12010]{12J10, 13A18, 13B22}
\date{}
\begin {abstract}
Let $(K,\nu)$ be an arbitrary valued field with valuation ring $R_\nu$ and $L=K(\al)$, where  $\al$ is a root of a monic irreducible polynomial $f\in R_\nu[x]$. In this paper, we characterize the integral closedness of $R_\nu[\al]$ in such a way our results  extend Dedekind’s criterion and a well-known theorem of Dedekind, without the assumption of separability of the extension $L/K$. We show also the  converse of that Dedekind’s theorem.
\end {abstract}

\maketitle

\section{{\bf Introduction}}\label{intro}

For a valued field  $(K, \nu)$, let $R_\nu$ be its  valuation ring, $M_\nu$ its  maximal ideal, $k_\nu=R_\nu/M_\nu$  its  residue field,  {$(K^h,\nu^h)$  a fixed  henselization of  $(K, \nu)$, and  $\ol{K}$  a fixed algebraic closure of $K$. We denote also by $\nu$ (resp. $\nu^h$) the Gaussian extension of $\nu$  (resp.  $\nu^h$) to the field $K(x)$  (resp. to  $K^h(x)$).
 Let  $\Ga_\nu$ be the  value group of $\nu$} and  $\Ga_\nu^+=\{ g\in \Ga_\nu,\,g>0\}$. As  $\Ga_\nu$ is a totally ordered group, if $\Ga_\nu^+$ has a smallest element, then it is unique and  we denote it by
 $\min(\Ga_\nu^+)$.   Let $f\in R_\nu[x]$ be a monic irreducible polynomial, $\al\in \ol{K}$  a root of $f$, $L=K(\al)$  the simple  extension of $K$ generated by $\al$, and {$R_L$}  the integral closure of $R_\nu$ in $L$. Denote by $\ol{f}$ the reduction   polynomial in $k_\nu[x]$ corresponding to $f$ and  consider  the  factorization   $\ol{f}=\prod_{i=1}^r \ol{\phi_i}^{l_i}$  into powers of monic irreducible polynomials in  $k_\nu[x]$, where  for every   $i=1, \dots, r$, $\phi_i\in R_\nu[x]$ is a monic lifting of $\ol{\phi_i}$. Recall that if $R_L$ is a free $R_\nu$-module, then it is  of rank $n=$deg$(f)$ as an $R_\nu$-module. In this case, the index of $f$ is the determinant of the transition matrix from an $R_\nu$-basis of $R_L$ to $(1,\al,\dots,\al^{n-1})$, and we denote the $\nu$-valuation of this determinant  by $ind(f)$. In particular, if $L/K$ is a finite separable extension and $\nu$ is a discrete rank-one valuation, then $R_L$ is a free $R_\nu$-module of finite rank.  
 For $R=\Z$,  $p$  a rational prime integer,  and $\nu_p$  the  $p$-adic valuation of the field $\q$, let $L/\q$ be a finite extension defined by a complex root $\al$ of a monic irreducible polynomial $f\in \Z[x]$. As $L/\q$ is separable,  $R_L$ is a free $\Z_p$-module of rank deg$(f)$. Furthermore,  $ind(f)$ is the $\nu_p$-valuation of the index of the finite group   $\Z_L/\Z[\al]$, where $\Z_L$ is the ring of integers of $L$, where $\Z_L$ is the ring of integers of $L$. \\
For a number field $K={\mathbb{Q} }(\alpha)$ generated by a complex root $\alpha$ of a monic irreducible polynomial $f(x)\in {\mathbb{Z} }[x]$, let ${\mathbb{Z} }_K$ be the ring of integers of $K$. For any prime integer $p$,   a well-known  theorem of  Dedekind says that  If $p$  does not divide the index $({\mathbb{Z} }_K:{\mathbb{Z} }[\alpha])$,  then the prime factorization of $p{\mathbb{Z} }_K$ has the same shape as the irreducible factorization of $\overline{f(x)}$ modulo $p$.
 In order to apply this  theorem effectively,
  in $1878$, Dedekind  { proved } the following  well known Dedekind's criterion, which allows to test if $p$  divides the index $({\mathbb{Z} }_K : {\mathbb{Z} }[\alpha])$. In 1977,  Uchida  proposed a new version which allows to test the integral closedness of $R[\alpha]$, where $R$ is a Dedekind ring. He used it to give a
 short proof of the monogecity of cyclotomic fields (see \cite{U}). In 1984, L\"uneburg  discovered another test. Unaware of Uchida's work, he  used his test to give an alternative proof of the monogecity of cyclotomic fields. In 2006, Ershov generalized this criterion to any arbitrary rank (see \cite{Er}).  In 2010, Khanduja and  Kumar gave a new  proof to Ershov's result in \cite[Theorem 1.1]{KK1}. In 2021, Jakhar and Khanduja gave a similar proof to Ershov's result in \cite[Theorem 1.1]{KK2}.
 In the papers  \cite{KK1, KK2}, the proofs of Theorem 1.1 are based on  the correspondence given in \cite[17.17]{En}, which assumes that this correspondence is valid under the condition of separability of the fields extension. In \cite{DBE}, El Fadil et al. gave another a new version of the Dedekind's criterion, which allows to characterize  the integral closedness of $R_\nu[\al]$. This new version computationally improves the results given in \cite{Er, KK1} in the context of  separability  of the extension
$L/K$. Besides,  Khanduja and  Kumar gave in  \cite{KK8} a converse of  Dedekind's theorem.
 Namely, for a Dedekind domain $R$ with quotient field $K$ and $L/K$ a finite  simple {\it separable} extension generated by  a root $\al$ of a monic irreducible polynomial $f\in R[x]$,  they proved that the condition $\pi$ does not divide
$\mbox{ind}(\alpha)$  is a necessary condition  for the existence   of exactly $r$
distinct prime ideals of $R_L$ lying above $\pi$, with generators,
 ramification indices, and residue degrees all as in  {  the} well-known  theorem of Dedekind, where $R_L$ is the integral closure of $R$ in $L$.
The aim of  this paper is that for any valued field $(K,\nu)$  and any finite simple extension $L/K$ generated by a root of  monic irreducible polynomial $f\in K[x]$, we show in Theorem \ref{Endler} that the valuations of $L$ extending $\nu$ are in one-to-one correspondence with the monic irreducible factors of $f$ in $K^h[x]$, where $K^h$ is a fixed henselization of $(K, \nu)$. The result is well-known when $f$ is separable (see \cite[17.17]{En} for instance).  However, we here drop this requirement. Furthermore, if $f\in R_\nu[x]$ is a monic irreducible polynomial, then Theorem \ref{main} gives a characterization of the integral closedness of  $R_\nu[\al]$. This is done with no separability assumption on $f(x)$ in contrast with \cite[Theorem 2.5]{DBE}. Theorem \ref{main} further improves \cite[Theorem 4.1]{KK1} in the sense that $K$ is not assumed to be Henselian. Theorem \ref{main 2} gives another characterization of the integral closedness of  $R_\nu[\al]$ via some information on valuations of $L$ extending $\nu$.
 Finally, we prove a new version based on the remainders of the Euclidean division  of $f(x)$ by  monic polynomials $\ph_1,\dots, \ph_r$ in $R_\nu[x]$, where $\ol{\ph_1},\dots, {\ph_r}$  are the monic irreducible divisors of $\ol{f(x)}$ in  $k_\nu[x]$. 
In addition, if $L/K$ is a separable extension,  $R_\nu$ is a discrete valuation  ring with maximal ideal  $M_\nu=(\pi)$, and   $\pi$ does not divide the index $[R_L:R_\nu[\al]]$, then thanks to a well-known  theorem of Dedekind,  the factorization of the ideal  $\pi R_L$ can be directly determined by the factorization of ${\overline f(x)}$ over $k_\nu$ (see \cite[ Chapter I, Proposition 8.3]{Neu}). Namely,  $\pi R_L=\prod_{i=0}^r \p_i^{l_i}$, where every $\p_i=pR_L+\phi_i(\al)R_L$ (so $f(\p_i)={\mbox{deg}}(\phi_i)$ is the residue degree of $\p_i$).
\section{{\bf Main Results}}
Keeping the above notations, denote by $(K^h, \nu^h)$ a fixed henselization of $(K, \nu)$, $R_{\nu^h}$ its valuation ring, $M_{\nu^h}$ its maximal ideal, and  $\ol{\nu^h}$ the unique extension of $\nu^h$ to the algebraic closure $\ol{K^h}$ of $K^h$.
If $L=K(\al)$ is a simple separable extension of $K$, defined by $\al\in \ol{K^h}$ a root of a monic irreducible polynomial $f\in K[x]$, then   a well-known important result,
\cite[17.17]{En},  asserts a one-one  correspondence between the valuations of $L$ extending $\nu$ and the monic irreducible factors of $f$ over $K^h$.
In \cite{KK1}, in order to show that the separability requirement of $ L / K $ is not necessary for the Dedekind's criterion given in \cite{Er}, the proof given by Khanduja and Kumar for  \cite[Theorem 1.1]{KK1} is based on \cite[Theorem 2.B]{KK1}, which was stated without proof and rather deduced from \cite[17.17]{En}, {which in turn assumes the separability of the extension.} {\it For this reason, we start this section by making clear  that the  correspondence in \cite[17.17]{En} remains valid  without  assuming  that $L/K$ is a separable extension.} {It is noted here that \cite{KK1} also contains a hint on part of our proof of the following result, but it does not give a complete proof. This hint was repeated in \cite{KK}}.
\begin{thm}\label{Endler}
Let $(K,\nu)$ be a valued field, $f\in R_\nu[x]$ be monic and irreducible over $K$, $L=K(\al)$ for some root $\al\in \ol{K^h}$ of $f$. Then there is a one-one  correspondence between the valuations of $L$ extending $\nu$ and the monic irreducible factors of $f$ over $K^h$. More precisely, if $f=\prod_{j=1}^t f_j^{a_j}$ is the monic irreducible factorization of $f$  in $K^h[x]$, then
every $a_j=1$ and there are exactly $t$ distinct valuations  $\om_1, \dots, \om_t$ of $L$ extending $\nu$. Moreover, for every $j=1, \dots, t$, if $\al_j$ is any root of $f_j$ in $\ol{K^h}$, then the valuation $\om_j$ corresponding to $f_j$ is precisely the valuation of $L$ satisfying: $\om_j(P(\al))=\ol{\nu^h}(P(\al_j))$ for any $P\in K[x]$.
\end{thm}
\begin{proof}
We only need to concern ourselves here with the inseparable case, as \cite[17.17]{En} deals with its separable counterpart. Assume that $f$ is inseparable over $K$. Then the characteristic of $K$ is a prime integer $p$. Let $d$ be the largest integer such that $f\in K[x^{p^d}]$. Then $f(x)=g(x^{p^d})$ for some separable polynomial $g$ over $K$. As $f$ is irreducible over $K$, so is $g$. Let $L_1=K(\al^{p^d})$ be the field extension of $K$, generated by $\al^{p^d}$ (a root of $g$){, and}  $g=\prod_{i=1}^t g_i$ be the factorization of $g$  in $K^h[x]$ with every $g_i$ is monic. As $g$ is separable, by \cite[17.17]{En}, there are exactly $t$ distinct valuations $v_1, \dots, v_t$ of $L_1$ extending $\nu$. Let $\om$ be a valuation  of $L$ extending $\nu$. Then its restriction $\om_{L_1}$ to $L_1$ is a valuation extending $\nu$. In order to show there are exactly $t$ distinct valuations $\om_1, \dots, \om_t$ of $L$ extending $\nu$, it suffices to show that for every valuation $\om_{L_1}$ of ${L_1}$ extending $\nu$,  there is only one valuation  $\om$ of ${L}$ extending $\om_{L_1}$. As $L/L_1$ is purely inseparable (because $x^{p^d}-\al^{p^d}$ is the minimal polynomial of $\al$ over $L_1$), there { is } a unique extension $\om$ of $\om_{L_1}$ to $L$ (see \cite[Corollaries 13.5 and 13.8]{En}). Since $g_i(x)$ is irreducible over $K^h$, for every $i=1,\dots,t$, then by the Capelli's Lemma (see \cite{FS}), $f_i(x)=g_i(x^{p^d})$ is irreducible over $K^h$ as well. Thus, $f=\prod_{j=1}^t f_j$ is the factorization of $f$  in $K^h[x]$. Hence for every $i=1,\dots,t$, $a_i=1$ and by uniqueness of the factorization of $f$ in $K^h[x]$, $f$ has exactly $t$ distinct monic irreducible factors over $K^h$.
This shows that there is a one-one  correspondence between the valuations of $L$ extending $\nu$ and the monic irreducible factors of $f$ over $K^h$. For the last statement of the theorem, it is easy to check that  for every $j=1,\dots,t$, the map defined by  $\om_j(P(\al))=\ol{\nu^h}(P(\al_j))$ for any $P\in K[x]$, defines a valuation on $L$. In fact, we have just to check the first property of a valuation, which is true because $f$ is the minimal polynomial of $\al_j$ over $K$. Moreover, as for every $i,j=1,\dots,t$, $i\neq j$ implies that $v_i\neq v_j$, then $\om_i\neq \om_j$.
\end{proof}
The next goal of this paper is to characterize the integral closedness of $R_\nu[\al]$.    
For this sake,
the following lemmas play key roles.
\begin{lem}\label{lemma} Keep the notations and assumptions as in Theorem \ref{Endler}, and let  $\ol{f}=\prod_{i=1}^r\ol{\ph_i}^{l_i}$ be the factorization into powers of monic irreducible polynomials of $k_\nu[x]$.
\begin{itemize}
\item[(i)] For every $i=1, \dots, r$, $\om_j(\phi_i(\al))>0$ for some valuation $\om_j$ of $L$ extending $\nu$.
\item[(ii)] For every valuation $\om_j$ of $L$ extending $\nu$ and every nonzero $P\in R_\nu[x]$, $\om_j(P(\al))\geq \nu^{h}(P)$.
\item[(iii)] For every $i=1, \dots, r$ and $\om_j$ a  valuation  of $L$ extending $\nu$ with $\om_j(\ph_i(\al))>0$, the equality holds in part (ii) if and only if  $\overline{\phi_i}$ does not divide $\overline{P_0}$, where $P_0=\frac{P}{a}$ with $a\in K$ and $\nu^h(a)=\nu^h(P)$.\\
In particular, $\om_j(\ph_i(\al))>0$ and if ${\mbox{ deg}}(P) < {\mbox{ deg}}(\phi_i)$, then $\om_j(P(\al))=\nu^{h}(P)$.
\end{itemize}
\end{lem}
\begin{proof}\hfill
\begin{itemize}
\item[(i)]
    Let $f=\prod_{j=1}^t {f_j}$ be the monic irreducible factorization of $f$ in $K^h[x]$. Since
    $\prod_{i=1}^r \ol{\phi_i}^{l_i} =\prod_{j=1}^t \ol{f_j}$, for every  fixed $i=1, \dots, r$ there is some $j=1, \dots, t$ such that $\ol{\phi_i}$ divides $\ol{f_j}$. Since $f_j$ is irreducible, it follows from Hensel's Lemma that $\ol{f_j}=\ol{\phi_i}^{s_i}$ for some $1\leq s_i \leq l_i$. Let $\al_j\in \ol{K^h}$ be a root of $f_j$. As $f_j(\al_j)=0$, we have $\ol{\phi_i(\al_j)}^{s_j} = \ol{f_j(\al_j)} =\ol{0}$ modulo $M_{\ol{\nu^h}}$. Thus, $\phi_i(\al_j)^{s_i}\in M_{\ol{\nu^h}}$ and so $\phi_i(\al_j)\in M_{\ol{\nu^h}}$. Thus, by Theorem \ref{Endler}, $\om_j(\phi_i(\al))=\ol{\nu^h}(\phi_i(\al_j))>0$ as desired.
\item[(ii)] Let $P_0=P/a$, where $a\in K$  and  $\nu^h(a)=\nu^h(P)$. As $\nu(P_0)=\nu^{h}(P_0)=0$, $P_0 \in R_\nu[x]$. Since $R_L = \bigcap_{j=1}^t R_{\om_j}$ (see \cite[Corollary 3.1.4]{En}), it follows that, for every $j=1, \dots, t$, we have $P_0(\al)\in R_\nu[\al]\subseteq R_L\subseteq R_{\om_j}$ and $$\qquad \quad \om_j(P(\al))=\om_j(a)+\om_j(P_0(\al))=\nu^h(a)+\om_j(P_0(\al))=\nu^{h}(P)+\om_j(P_0(\al)) \geq \nu^{h}(P).$$
\item[(iii)] Define the map $\psi: k_\nu[x] \to R_{\om_j}/M_{\om_j}$ by $\overline{P} \mapsto P(\al)+ M_{\om_j}$. Since $M_\nu \subseteq M_{\om_j}$, $\psi$ is a well-defined ring homomorphism. As $\om_j(P(\al))=\nu^{h}(P) +\om_j(P_0(\al))$ (see part (ii)), it follows that $\om_j(P(\al))=\nu^{h}(P)$ if and only if $\om_j(P_0(\al))=0$, if and only if $P_0(\al)\in R_{\om_j} - M_{\om_j}$, if and only if $\overline{P_0} \not\in \mbox{ker}\,\psi$.
As $\om_j(\phi_i(\al))>0$, $\phi_i(\al)\in M_{\om_j}$ and so
$\overline{\phi_i}(X) \in \mbox{ker}\,\psi$.
Since  $\overline{\phi_i}$
 is irreducible over $k_\nu$, $\mbox{ker}\,\psi$ is the principal ideal of   $k_\nu[x]$  generated by $\overline{\phi_i}$.
  It now follows that $\om_j(p(\al))=\nu^{h}(P)$ if and only if
$\overline{\phi_i}(X)$ does not divide $\overline{P_0}(X)$.
\end{itemize}
\end{proof}
\begin{lem}\label{inf}
Let {$f\in R_{\nu^h}[x]$} be a monic irreducible polynomial and $L=K^h(\al)$, where $\al\in \ol{K}$ is a root of {$f$}. Assume that  {$\ol{f}=\ol{\ph}^l$} for some monic polynomial $\ph\in R_{\nu^h}[x]$, whose reduction modulo $M_{{\nu^h}}$, is irreducible in $k_\nu[x]$. Let  {$q, r\in R_{\nu^h}[x]$} be, respectively, the quotient and the remainder upon the euclidean division of {$f$}
by $\phi$. Then
 $R_{\nu^h}[\al]$ is integrally closed if and only if $l=1$ or $\Ga_\nu^+$ has a minimal element $\s$ and $\nu^h(r)=\s$.
 \end{lem}
 \begin{proof}
 {To begin with, recall that $(K,\nu)$ and $(K^h,\nu^h)$ have the same value group and the same residue field.
Assume that $l=1$ or $\Ga_\nu^+$ has a minimal element $\s$ with $\nu^h(r)=\s$, we seek to show that  $R_{\nu^h}[\al]$ is integrally closed in either case.}
\begin{itemize}
\item [(1)]
Assume that $l=1$. An arbitrary element of $L$ is on the form  $\t=P(\al)/b$ for some $b\in R_{\nu^h}$ and $P\in R_{\nu^h}[x]$, with {$\mbox{deg}(P)<\mbox{deg}(f)$.
Since  $\mbox{deg}(P)<\mbox{deg}(\ol{f})$} and $l=1$, $\ol{\phi}$ does not divide $\ol{P}$. { Thus by Lemma \ref{lemma}, $\om(P(\al))=\nu^h(P)$ and   $\om(\phi(\al))>0$},  { where $\om$ is the unique valuation of $L$ extending $\nu^h$}.
It follows that { if $\t\in R_L$, then   $\om(\t)\ge 0$ and so, $0\le \nu^h(b)\le \om(P(\al))=\nu^h(P)$}. Hence $b$ divides  $P$ in $R_{\nu^h}[x]$, and $\t \in R_{\nu^h}[\al]$. Thus $R_{\nu^h}[\al]$ is integrally closed in this case.
\item [(2)]
Now assume that  $l\ge 2$. { We show that if  $\nu^h(r)$ is not the } minimal element of $\Ga_\nu^+$, then $\t=\frac{\ph(\al)^{l-1}}{b}\in R_L$ for an adequate element $b\in R_{\nu^h}$ with $\nu^h(b)>0$ and  so $\t\not \in R_{\nu^h}[\al]$ (because  as $\ph$ is monic, $\frac{\ph(x)^{l-1}}{b}\not\in R_{\nu^h}[x]$). Let {$f(x)=a_{l}(x)\phi^{l}(x)+\cdots+a_0(x)$}  be the $\ph$-expansion of $f$.  Since  $\ol{f}=\ol{\ph}^l$,  {$a_l=1$ and} for every $i=0,\dots,l-1$,   $\ol{a_i}=0$, $\nu^h(a_i)>0$, and $\nu^h(a_i)\in \Ga_\nu^+$.
   Let $\tau=\mbox{ min }(\nu^h(a_i),\,i=0,\dots,l-1)$, {$v\in \Ga_\nu^+$ with $v<\nu^h(r)$,
$\d=\min \{\tau, \nu^h(r)-v\}$,    $b\in R_{\nu^h}$ be such that $\nu^h(b)=\d$, and
$\t=\frac{\ph(\al)^{l-1}}{b}$}.
 To conclude that $\t \in R_L$, it suffices to show that $\om(\t)\ge 0$ for the unique valuation  $\om$ of $L$ extending $\nu^h$.
 { Recall that as  for every $i=0,\dots,l-1$, deg$(a_i)<$deg$(\ph)$,  by Lemma \ref{lemma},  $\om(a_i(\al))=\nu^h(a_i)\ge \s$. As $r=a_0$, $\nu^h(r)\in \Ga_\nu^+$}.
There are two cases:  The first one   $\om(\phi(\al))> \d$,    $\om(\ph(\al)^{l-1})> (l-1)\d\ge \d$ (because $l\ge2$). Thus $\om(\t)\ge 0$ and $\t\in R_\om$.\\
For the second one; $\om(\phi(\al))\leq \d$. As {$\t=-(\frac{a_{l-1}(\al)}{b}\ph^{l-2}(\al)+\cdots+\frac{a_1(\al))}{b}-\frac{a_{0}(\al)}{b\ph(\al)})$} and  $\nu^h(b)=\d\le \tau \le \nu^h(a_i)=\om(a_i(\al))$ for every $i=1,\dots,l-1$, $\om(\t)\ge 0$ if and only if
$\om(\frac{a_{0}(\al)}{b\ph(\al)})\ge 0$.  As $\d\le \tau$ and  $\d\le  \nu^h(r)-v$, $2\d\le  \nu^h(r)$ and $\om(\frac{a_{0}(\al)}{b\ph(\al)})\ge \nu^h(r)- \d-\om(\ph(\al))\ge \nu^h(r)- 2\d\ge 0$.   {In both cases $\om(\t)\ge 0$. Thus} $\t\in R_L$.
\end{itemize}
Conversely,  assume that $l\ge 2$ and {$\s=\mbox{ min }(\Ga_\nu^+)=\nu^h(r)$.   We claim that  $\om(\ph(\al))=\frac{\s}{l}$, where $\om$ is the unique valuation of $L$ extending $\nu$ (by Lemma \ref{lemma}). Indeed, as $\om(\ph(\al))>0$, and for every $i=1,\dots, l-1$, $\om(a_i(\al))=\nu^h(a_i)\ge \s$, we have $\s=\nu^h(r)=\nu^h(a_0)=\om(a_0(\al))$ and so, $\om(a_i(\al)\ph^i(\al))>\s=\om(a_0(\al))$. On the other hand,  $\ph^l(\al)+{a_{l-1}(\al)}\ph^{l-2}(\al)+\cdots+{a_1(\al))}=-{a_{0}(\al)}=-r(\al)$ implies that  $\om(\ph^l(\al))=\om(a_0(\al))=\s$. Thus  $\om(\ph(\al))=\frac{\s}{l}$ as desired}.\\
Let $P\in R_{\nu^h}[x]$ with $\nu^h(P)=0$,  {$\mbox{deg}(P)<\mbox{deg}(f)$, $\t=\frac{P(\al)}{b}$, and show that if $\t\in R_\om$, then $\t\in R_{\nu^h}[\al]$}. Let $u=V_{\ol{\ph}}(\ol{P})$ be the $\ol{\ph}$-adic valuation of $\ol{P}$. { If $u=0$, then $\ol{\ph}$ does not divide $\ol{P}$, and thus $\om(P(\al))=\nu^h(P)=0$. Otherwise,} as $\mbox{deg}(P)<\mbox{deg}(f)$, then $u <l$ and  $P=Q\ph^u+T$, where  $\ol{\ph}$ does not divide $\ol{Q}$ and $\ol{T}=0$. So, $\nu^h(T)\ge \s$. Let $T(x)=\sum_{i=0}^sb_i(x)\ph^i(x)$ be the $\ph$-expansion of $T$.
{As $\ol{T}=0$, for every $i=1,\dots,s$,
 $\ol{b_i}=0$ and $\nu^h(b_i)\ge \s$. As  $\om(\ph(\al))>0$,
then  for every $i=1,\dots,s$,
 $\om(b_i(\al)\ph^i(\al))>\s$. Thus
   $\om(T(\al))\ge \om(b_0(\al))=\nu^h(b_0))\ge \s$}.
 Also {as $u=V_{\ol{\ph}}(\ol{P})$, $\ol{\ph}$ does not divide  $\ol{Q}$ and so by Lemma \ref{lemma}, $\om(Q(\al))=\nu^h(Q)=0$ (because $Q$ is monic)}.
Thus $\om(Q(\al)\ph^u(\al))=u\om(\ph(\al))=\frac{u\s}{l}<\s$, then $\om(P(\al))=\frac{u\s}{l}<\s$.\\
 {It follows that in both cases, for every $P\in R_{\nu^h}[x]$ of degree less than degree$(f)$, if $\nu^h(P)=0$, then $\om(P(\al))=\frac{u\s}{l}<\s$.
Hence if $\t \in R_L=R_\om$, then $\nu^h(b)\le \om(P(\al))<\s$. Thus $\nu^h(b)=0$; $b$ is invertible in $R_{\nu^h}$. Hence $b$  divides $P$ in $R_{\nu^h}[x]$} and $R_L=R_{\nu^h}[\al]$ as claimed.
\end{proof}
\begin{cor}\label{phi}
{Under the hypotheses and notations of Lemma \ref{inf}, if $l\ge 2$ and $R_{\nu^h}[\al]$ is integrally closed, then $\Ga^+$ has a minimal element $\s$ and $\om(\ph(\al))=\frac{\s}{l}$, where $\om$ is the unique valuation of $K^h(\al)$  extending $\nu^h$.}
\end{cor}
\begin{proof}
{Suppose that $R_{\nu^h}[\al]$ is integrally closed.
Since  $l\ge 2$, it follows from  Lemma \ref{inf} that  $\Ga_\nu^+$ has a minimal element $\s$ with $\s=\nu^h(r)$. we show that   $\om(\ph(\al))=\frac{\s}{l}$. Let $f=\ph^l+b_{l-1}\ph^{l-1}+\dots+b_0$ be the $\ph$-expansion of $f$. Then $r=b_0$. As $\ol{f}=\ol{\ph}^l$, $\ol{b_i}=0$  and so $\nu^h(b_i)\ge \s$  for every $i=0,\dots,l-1$. By Lemma \ref{lemma}, $\om(\ph(\al))>0$. Thus, $\om({b_i}(\al)\ph(\al)^i)>\s$  for every $i=1,\dots,l-1$. It follows that $\mbox{ min }\{\om({b_i}(\al)\ph(\al)^i),\, i=0,\dots,l\}=\mbox{min}(\om({b_0}(\al)),l\om(\ph(\al))$. If $\om({b_0}(\al))\neq l\om(\ph(\al))$, then $\om(f(\al))=\mbox{min}(\om({b_0}(\al)),l\om(\ph(\al))$, which is impossible because $f(\al)=0$. So, $l\om(\ph(\al)=\om({b_0}(\al))=\om(r(\al))=\nu^h(r)=\s$ (by Lemma \ref{lemma})}.
 \end{proof}

Now we get to our second main result, which  does not require the separability of the extension required in \cite[Theorem 2.5]{DBE} and also improves  \cite[Theorem 4.1]{KK1} in the sense that the base field is not assumed to be Henselian.
Recall that in 2006, Ershov generalized the Dedekind's criterion to any arbitrary rank (see \cite{Er}).  In 2010, Khanduja and  Kumar gave a new  proof to Ershov's result in \cite{KK1}. Also in 2021, Jakhar and Khanduja gave a similar proof to Ershov's result in \cite{KK2}.
But the proofs of Theorem 1.1 in  \cite{KK1, KK2} are based on  the correspondence given in \cite[17.17]{En}, which assumes that this correspondence is valid under the condition of separability of the fields extension. So we taught that these proofs are at least incomplete. In this paper, based on the remainders of the Euclidean division of $f(x)$ by some monic polynomials, we propose a new version of Dedekind's criterion for arbitrary rank valued fields. 
 \begin{thm}\label{main}
{ Let  $L=K(\al)$ be the simple extension of $K$ generated by $\al\in\overline{K}$ a root of a monic irreducible polynomial $f\in R_\nu[x]$     and assume that $\ol{f}=\prod_{i=1}^r \ol{\phi_i}^{l_i}$  is  the factorization of $\ol{f}$ into powers of monic irreducible polynomials in  $k_\nu[x]$, where  for every   $i=1, \dots, r$, $\phi_i\in R_\nu[x]$ is a monic lifting of $\ol{\phi_i}$. Let $I=\{1\leq i \leq r\,|\, l_i\geq 2\}$ and for every $i\in I$, let $(q_i,r_i)$ be the quotient and the remainder of the euclidean division of $f(x)$  by $\ph_i(x)$. Then  the following statements are equivalent:
\begin{itemize}
\item[(i)]
$R_\nu[\al]$ is integrally closed.
\item[(ii)]  Either $I=\emptyset$  or  $\nu(r_i)=\min(\Ga_\nu^+)=\s$ for every $i\in I$.
\end{itemize}}
\end{thm}
\begin{proof}
  As $\ol{f}=\prod_{i=1}^r\ol{\ph_i}^{l_i}$ is the factorization of $\ol{f}$ in $k_\nu[x]$, by Hensel's lemma, $f=\prod_{i=1}^rF_i$ in $R_{\nu^h}[x]$ such that for every $i=1,\dots,r$,  $\ol{F_i}=\ol{\ph_i}^{l_i}$.
For  $i=1, \dots, r$, let $\al_i\in{\ol{K}}$ be a root of $F_i$.
As $\ol{F_i}$ and $\ol{F_j}$ are coprime, $\nu^h(Res(F_i, F_j))=0$.  \\
$(i)\Longrightarrow (ii)$:
Assume that $R_\nu[\al]$ is integrally closed and $I\neq \emptyset$. We need to show that $\Gamma_\nu^+$ has a minimal element $\s$ and,  for every $i\in I$, $\nu(r_i)=\min(\Ga_\nu^+)$. Let $i\in I$.
If  $F_i$ factors in $R_{\nu^h}[x]$ as $F_i=F_{i1}F_{i2}$, where every $F_{ij}$ is not a constant polynomial then, as $\ol{F_{i1}}$ and $\ol{F_{i2}}$ are two powers of $\ol{\ph_i}$ and so, $\nu^h(Res(F_{i1}, F_{i2}))>0$. By using the index formula given in \cite{KK17}, $ind(F_i)\ge\nu^h(Res(F_{i1}, F_{i2}))>0$, which  is a contradiction because the assumption that $R_\nu[\al]$ is  integrally closed implies that
$ind(f)=0$. Hence for every $i=1,\dots,r$,  $F_i$ is irreducible over $K^h$ and  $ind(F_i)=0$.
By Lemma \ref{inf}, for every $i\in I$, $\nu(r_i)=\min(\Ga_\nu^+)$. \\
$(ii)\Longrightarrow (i)$: Assume that either $I=\emptyset$  or  $\nu(r_i)=\min(\Ga_\nu^+)$ for every $i\in I$. We show that $R_\nu[\al]$ is integrally closed. Let $\theta=P(\al)/b$, where $P\in R_\nu[x]$ is a  polynomial with $\deg(P)<\deg(f)$, $\nu(P)=0$, and $b\in R_\nu$.  We show that $\theta\in\R_L$ implies that $\nu(b)=0$, and so $b$ divides $P$ in $R_\nu[x]$ and thus $\theta \in R_\nu[\al]$. Since $\deg(P)< \deg(f)$, there exists $i=1,\dots,r$ such that $\nu_{\ol{\ph_i}}(\ol{P})=k_i<l_i$, where $\nu_{\ol{\phi_i}}$ is the $\ol{\phi_i}$-adic valuation.  By  Lemma \ref{lemma}, let $\om$ be a valuation  of $L$ extending $\nu$ such that $\om(\ph_i(\al))>0$. It follows that if $k_i=0$, then  $\ol{\ph_i}$ does not divide $\ol{P}$ and so by  Lemma \ref{lemma}, $\om(P(\al))=\nu^h(P)=\nu(P)=0$. If $\theta\in R_L$, then $\theta\in R_{\om}$, and so $0\le \om(\theta)=\om(P(\al))-\nu^h(b)$. Thus $\nu(b)=\nu^h(b)=0$.  In particular, if $l_i=1$, then $\nu(b)=0$ and $b$ divides $P$ in $R_\nu[x]$.
If $l_i\ge 2$ and  $k_i\ge 1$, then let $P=Q_i\ph_i^{k_i}+T_i$, with $Q_i, T_i\in R_\nu[x]$ such that $\ol{\ph_i}$ does not divide $\ol{Q_i}$ and $\nu(T_i)\ge \s$. Since $\nu(r_i)=\s$, it follows from Corollary \ref{phi} and its proof that $\om(\phi_i(\al))=\frac{\s}{l_i}$. Since $\om(Q_i(\al)\ph_i^{k_i}(\al))= \om(Q_i(\al))+k_i\om(\ph_i(\al))=0+\frac{k_i\s}{l_i}<\s$ and $\om(T_i(\al))\ge \nu^h(T_i)\ge  \s$, we get  $\om(P(\al))=\frac{k_i\s}{l_i}<\s$, and so $\om(P(\al))=0$. Therefore, $\nu(b)=\nu^h(b)=0$ and $b$ divides $P$ in $R_\nu[x]$.
\end{proof}
Our third main result, Theorem \ref{main 2} bellow gives  a {characterization of the integral closedness of $R_\nu[\al]$ via some information on distinct valuations of $L$ extending $\nu$}. Remark that the statement of this theorem appears in  \cite[Theorem 1.3, p:17]{RH}. 
But as it was write before \cite[Theorem 1.3, p:17]{RH}, its  proof was based on the correspondence in [3,17.17], which   assume   that $L/K$ is a separable extension.  So this Theorem \ref{main 2} is an application of our  correspondence in Theorem \ref{Endler}, in which   the assumption of separability of the extension  is not required.
 \begin{thm}\label{main 2}
Keep the assumptions {and notations} of Theorem \ref{main}. The following statements
 are equivalent:
\begin{itemize}
\item[(i)] $R_\nu[\al]$ is integrally closed.
\item[(ii)]
 There are exactly $r$ distinct valuations $\om_1, \dots, \om_r$ of $L$ extending $\nu$,
and if $I\neq\emptyset$, then $\Gamma_\nu^+$ has a minimal element $\s$ and  for every $i\in I$,  $l_i\om_i(\phi_i(\al))=\s$, where $\om_i$ is a valuation satisfying
$\om_i(\phi_i(\al))>0$.
\end{itemize}
\end{thm}
\begin{proof}
 { By Hensel's lemma, let
 $f=F_1\cdots F_r$ in $R_{\nu^h}[x]$  { be} such that for every $i=1,\dots,r$,
$\overline{F_i}=\overline{\phi_i}^{l_i}$.\\
 $(i)\,\Longrightarrow\, (ii)$: Assume that $R_\nu[\al]$ is integrally closed and show that every $F_i$ is irreducible over $K^h$, from  which it would follow, {by Theorem \ref{Endler}, that $r=t$} and there are exactly $r$ distinct valuations $\om_1, \dots, \om_r$ of $L$ extending $\nu$. By the  argument presented in the proof  of Theorem \ref{main},  $ind(f)$ is well defined and equals  $0$. Thus for every $i=1,\dots,r$,    $ind(F_i)=0$ and so, as was shown in the proof of Theorem \ref{main}, if $i\in I$, then $F_i$ is irreducible over $K^h$.
  {If $i\not\in I$, then  $\ol{F_i}=\ol{\ph_i}$ is irreducible over $k_\nu$, and so  $F_i$ is irreducible over $K^h$.
   Therefore, there are exactly  $r$ distinct valuations  $\om_1,\dots,\om_r$ of $L$ extending $\nu$.
Moreover, if $I\neq\emptyset$, then for every  $i\in I$, let   $\al_i\in \ol{K^h}$ be a root of $F_i$ and $\om_i$  the valuation of $L$ associated to the irreducible factor $F_i$. Then $\om_i(\ph_i(\al))=\ol{\nu^h}(\ph_i(\al_i))=V_i(\ph_i(\al_i))$, where $V_i$ is the unique valuation of $K^h(\al_i)$}.
%If $\om_i(\ph_i(\al))>0$, then as by assumption, $\ol{F_i}=\ol{\ph_i}^{l_i}$, we have $\om_i$ is the associated valuation of $L$ to the irreducible factor $F_i$.
  As $R_\nu[\al]$ is integrally closed, $F_i$ is irreducible over $K^h$ and $R_{\nu^h}[\al_i]$ is integrally closed. Hence, by Corollary \ref{phi}, $\om_i(\ph_i(\al))=\frac{\s}{l_i}$ and $l_i\om_i(\ph_i(\al))=\s$.\\
   $(ii)\,\Longrightarrow\, (i)$:  By Theorem \ref{main}, it suffices to show that if $I\neq \emptyset$, then for every $i\in I$, $\nu(r_i)=\s$.
Assume that $I\neq \emptyset$ and for every $i\in I$, $l_i\om_i(\phi_i(\al))=\s$  is the minimum element of $\Ga_\nu^+$ when $\om_i$ satisfies
$\om_i(\phi_i(\al))>0$. For such an $i$,} let $f=Q\ph_i^{l_i}+h\ph_i+r_i$
 for some $Q,h\in R_\nu[x]$.  {As ${l_i}\ge 2$,   $\ol{\ph_i}$ does not divide $\ol{Q}$ and, by Lemma \ref{lemma},  $\om_i(Q(\al))=\nu(Q)=0$ and so $\om_i(Q(\al)\ph_i(\al)^{l_i})=\s$. Also, as $l_i\ge 2$,  $\ol{h}=0$ and so $\nu(h)\ge \s$. Thus,  $\om_i(h(\al)\ph_i(\al))>\om_i(h(\al))\ge \s$. So, $\om_i(Q(\al)\ph_i(\al)^{l_i})\neq \om_i(h(\al)\ph_i(\al))$. As $f(\al)=0$, $\nu(r_i)=\om_i(r_i(\al))=\mbox{min}(\om_i(Q(\al)\ph_i^{l_i}(\al)), \om_i(h(\al)\ph_i(\al)))=\s$.}
Therefore, $R_\nu[\al]$ is integrally closed.
\end{proof}
Under notations of Theorem \ref{main 2}, for any valuation $\om$ of
 $L$ extending $\nu$, we denote the ramification index $[\Ga_\om:\Ga_\nu]$ by $e(\om/\nu)$
 and the residue degree $[k_\om:k_\nu]$ by $f(\om/\nu)$.
In the following {remark}, we calculate  $e(\om/\nu)$ and $f(\om/\nu)$ for every  such   $\om$ when $R_\nu[\al]$ is integrally closed. In particular, if $\nu$ is a rank-one valuation, Theorem \ref{Dedek} is a  reformulation of Theorem \ref{main 2} by using prime ideal factorization instead of extension of valuations.
 \begin{rem}
 Under the hypothesis of Theorem \ref{main 2}, if $R_\nu[\al]$ is integrally closed, then there are exactly $r$ distinct valuations $\om_1,\dots, \om_r$ of $L$ extending $\nu$ and for every $i=1\dots,r$,  {$e(\om_i/\nu)= l_i$ and $f(\om_i/\nu)= m_i=$deg$(ph_i)$}. Indeed, let $i=1,\dots,r$. As $\frac{\s}{l_i}=\om_i(\ph_i(\al))$, $\frac{\s}{l_i}\in  \Ga_{\om_i}$. Let $\Ga_\nu(\frac{\s}{l_i})$ be the group generated by $\Ga_\nu$ and $\frac{\s}{l_i}$. Then the index $l_i=[\Ga_\nu(\frac{\s}{l_i}):\Ga_\nu]$   divides   $[\Ga_{\om_i}:\Ga_\nu]$ and so $e(\om_i/\nu)\ge l_i$. Moreover, let $\iota\,: k_\nu[x]\longrightarrow k_{\om_i}$ be the map defined by $\iota(\ol{P})= P(\al)+M_{\om_i}$, where $k_{\om_i}$ is the residue field of $\om_i$ and $M_{\om_i}$ is its  maximal ideal. Then the kernel of $\iota$ is generated by $\ph_i$. { Let $\F_{\ph_i}$ be the field    $\frac{ k_\nu[x]}{(\ol{\ph_i})}$. Then, } $k_{\nu}\subset \F_{\ph_i}\subset k_{\om_i}$. So,  $[\F_{\ph_i}:k_\nu]$ divides $f(\om_i/\nu)$. This means  that $m_i$ divides $f(\om_i/\nu)$  and so $f(\om_i/\nu)\ge m_i$. But as for every
 $i$, $l_i\cdot m_i = deg(F_i)$, we get the claimed equalities. We note here that the extension $L/K$ is defectless since $[L:K]=\sum_{i=1}^rl_im_i=\sum_{i=1}^re_if_i$.
 \end{rem}

 Let $R$ be a Dedekind domain, { $K$ the quotient field of $R$, } $L=K(\al)$  a finite  simple field  extension  {generated by a root  $\al\in \ol{K}$ of a monic irreducible polynomial  $f\in R[x]$}, and $R_L$  the integral closure of $R$ in $L$. Let $\p$ be a nonzero prime ideal of
 $R$, $R_\p$ the localization of $R$ at $\p$, and $\nu=\nu_\p$  the $\p$-adic discrete valuation of
$K$. Let $\ol{f}=\prod_{i=1}^r \ol{\phi_i}^{l_i}$ be the
 factorization of $\ol{f}$ into powers of distinct monic irreducible factors modulo $\p$,
 and $\phi_i\in R[x]$  a monic lifting of $\ol{\phi_i}$ for $i=1, \dots, r$. If $L/K$ is separable, then a well-known Theorem  of Dedekind says that if $\p$ does not divide the index ideal
$\mbox{ind}(\al)=[R_L:R[\al]]$, then $\p$ factors in $R_L$ as $\p R_L= \prod_{i=1}^r \Q_i^{l_i}$, where $\Q_i$ are distinct prime ideals of $R_L$ with $\Q_i=\p R_L+\phi_i(\al) R_L$ and
 $f(\Q_i/\p)=\mbox{deg}(\phi_i)$. Khanduja and Kumar gave the converse of Dedekind's theorem in \cite[Theorem 1.1]{KK8} under the  assumption that   $L/K$ is  separable.  Our fourth main result relaxes the separability of the extension $L/K$ required in both  Dedekind's theorem and  \cite[Theorem 1.1]{KK8}.
 \begin{thm}\label{Dedek}
Under the assumptions and notations  as above, $R_{\p}[\al]$ is integrally closed if and only if $\p R_L$ factors into powers of prime ideals of $R_L$ as $\p R_L=
\prod_{i=1}^r \Q_i^{l_i}$, where $\Q_1,\dots,\Q_r$ are the distinct prime ideals of $R_L$ lying above $\p$, {and} for every $i=1, \dots, r$,
$\Q_i=\p R_L+\phi_i(\al) R_L$.
\end{thm}
\begin{proof}
{To begin with, we recall  that the condition $R_\p[\al]$ is integrally closed is equivalent to $R_\p[\al] = (R_L)_\p$, where  $R_L$ is the integral closure of $R$ in $L$.}
Let $\p R_L= \prod_{i=1}^t {\Q_i}^{e_i}$ be the factorization of $\p R_L$ into powers of prime ideals of $R_L$. { We } show that   $R_{\p}[\al]$ is integrally closed if and only if $t=r$ and for every $i=1, \dots, r$,
$\Q_i=\p R_L+\phi_i(\al) R_L$.\\
 Let ${\mathcal P}=\{\Q_1,\dots,\Q_t\}$ be the set of all prime ideals of $R_L$ lying above $\p$ and  $Ext_\nu(L)$  the set of all valuations of $L$ extending $\nu$. For every  $\om\in Ext_\nu(L)$, let $R_\om$ be the  valuation ring of $\om$ and $M_\om$ its maximal ideal. As  $M_\om\cap R_L$ is a prime ideal of $R_L$ lying above $\p$, $M_\om\cap R_L=\Q_i$, $\om=\frac{1}{e_i}\nu_{\Q_i}$, and  $e_i=|\Ga_\om/\Z|$ is the ramification index of $\om$ for some $i=1,\dots,t$, where $\nu_{\Q_i}$ is the $\Q_i$-adic valuation. It follows that $\iota \, : {\mathcal P}\longrightarrow Ext_\nu(L)$ defined  by  $\iota(\Q_i)=\frac{1}{e_i} \nu_{\Q_i}$ is a one-one correspondence.
 Thus, by Theorem \ref{main 2}, ${(R_L)}_\p=R_\p[\al]$ if and only if $t=r$ and, for every
$i=1,\dots,r$, $l_i\ge 2$ implies that $\nu_{\Q_i}(\phi_i(\al))=l_i\cdot \frac{1}{e_i}$ (the  last equality follows from Corollary \ref{phi} because $1$ is the minimal element of $\Ga_\nu^+=\Z^+$).
By the previous Remark, if $R_{\p}[\al]$ is integrally closed { and $l_i\ge 2$,} then  $e_i=l_i$. Thus $\nu_{\Q_i}(\phi_i(\al))=e_i\om_i(\phi_i(\al)))=1$. { In this case,  for $j=1,\dots,r$ with $j\neq i$, we have  $\nu_{\Q_j}(\phi_i(\al))=0$ and $\nu_{\Q_i}(\phi_i(\al))=1$. Thus,   $\Q_i=\p R_L+\phi_i(\al) R_L$.} Finally,
$R_{\p}[\al]$ is integrally closed if and only if
 for every $i=1,\dots,r$, $l_i\ge 2$ implies that $e_i=l_i$ and  $\Q_i=\p R_L+\phi_i(\al) R_L$. { For $l_i=1$, let  $n=$deg$(f)$, $m_i=$deg$(\ph_i)$, and $f_i=f(\Q_i/\p)$ be  the residue degree of $\Q_i$. As $\sum_{i=1}^r{e_if_i}=n=\sum_{i=1}^r{l_im_i}$ and for every $i=1,\dots,r$, $e_i\ge 1$ and $f_i\ge m_i$ (because $m_i$ divides $f_i$),  if $l_i=1$, then $e_i=1$ and $f_i=m_i$. Thus,  $R_{\p}[\al]$ is integrally closed if and only if $\p R_L= \prod_{i=1}^r \Q_i^{l_i}$, where for every $i=1,\dots,r$, $l_i\ge 2$ implies
 that $\Q_i=\p R_L+\phi_i(\al) R_L$.}
 Recall that  for every $i\neq j$,  $\nu_{\Q_j}(\phi_i(\al))=0$ and $\nu_{\Q_i}(\phi_i(\al))>0$. Thus if $e_i=1$, then $\Q_i=\p R_L+\phi_i(\al) R_L$ regardless of whether  $\nu_{\Q_i}(\phi_i(\al))=1$ or not.
\end{proof}

\section*{Acknowledgement}
Based on personal communications with Professor Yuri Ershov, he kindly shared with us the news that he independently has practically the same proof of Theorem \ref{Endler}. We deeply thank him for the fruitful discussion and encouragement.

\end{document}